\documentclass[12pt]{amsart}
\usepackage{a4wide}
\usepackage{float}
\usepackage{euscript,eufrak,verbatim}
\usepackage{graphicx}
\usepackage[usenames]{color}
\usepackage[colorlinks,linkcolor=red,anchorcolor=blue,citecolor=blue]{hyperref}
\usepackage{amsmath}
\usepackage{amsthm}
\usepackage[all]{xy}
\usepackage{graphicx} 
\usepackage{amssymb} 
\usepackage{stmaryrd}

\newtheorem*{theomain}{Main Theorem}
\newtheorem{thm}{Theorem}
\newtheorem{prop}[thm]{Proposition}
\newtheorem{df}[thm]{Defintion}
\newtheorem{lem}[thm]{Lemma}

\newtheorem{coro}[thm]{Corollary}
\newtheorem*{cor}{Corollary}

\newtheorem{ques}[thm]{Question}

\newtheorem{rem}[thm]{Remark}

\def\N{\mathbb{N}}
\def\Z{\mathbb{Z}}
\def\N{\mathbb{N}}
\def\R{\mathbb{R}}

\def\FF{\mathcal{F}}

\def\MM{\mathcal{M}}

\def\wdim{\text{\rm Widim}}

\def\mdim{\text{\rm mdim}}
\def\supp{\text{\rm supp}}
\def\stab{\text{\rm stabdim}}

\def\diam{\text{\rm diam}}
\def\AA{\mathcal{A}}
\def\BB{\mathcal{B}}

\makeatletter 
\@addtoreset{equation}{section}
\numberwithin{equation}{section}

\makeatletter
\@namedef{subjclassname@2020}{\textup{2020} Mathematics Subject Classification}
\makeatother

\title{Topological mean dimension of induced systems }

	\author{David Burguet}
	\address
	{Sorbonne Universite, LPSM, 75005 Paris, France}
	\email{david.burguet@upmc.fr}
	\author{Ruxi Shi}
\address
	{Sorbonne Universite, LPSM, 75005 Paris, France}
\email{ruxi.shi@upmc.fr}

\subjclass[2020]{}
\begin{document}

	\begin{abstract}
For a topological system with positive topological entropy, we show that the induced transformation on the set  of probability measures endowed with the weak-$*$ topology has infinite topological mean dimension.  We also estimate the rate of divergence of the entropy with respect to the Wasserstein distance when the scale goes to zero. 

	\end{abstract}

	\maketitle
\section{Introduction}	

Let $(X, T)$ be a topological dynamical system, i.e. $X$ is a compact  metrizable space and $T: X\to X$ is a continuous map.  Let $\MM(X)$ be the space of Borel probability measures on $X$ endowed with the weak-$*$ topology. Then the push-forward map $T_*:  \mu \mapsto \mu\circ T^{-1}$ is a continuous map on $\MM(X) .$ 
The topological system $(\MM(X), T_*)$ is called the \emph{induced sytem} of $(X,T)$ on probability measures. 

W. Bauer and K. Sigmund \cite{Sigmun} have investigated which  dynamical properties of $(\MM(X),T_*)$ are inherited from $(X,T)$. In particular they obsversed that $h_{top}(T_*)$ is infinite when $h_{top}(T)$ is positive. Later  B. Weiss and E. Glasner \cite{glasner1995quasi}  showed that if $(X,T)$ has zero topogical entropy, then so does $(\MM(X),T_*)$. Therefore we have the following equivalences :
$$h_{top}(T)>0\Leftrightarrow h_{top}(T_*)>0 \Leftrightarrow h_{top}(T_*)=\infty.$$

M. Gromov \cite{G} has introduced a new topological invariant of dynamical systems called {\it  (topological) mean  dimension} 
	as a dynamical analogue of topological covering dimension. In some sense the mean dimension $\mdim(T)$ estimates the number of parameters that we need to describe the system $(X,T)$ per unity  of time.   A dynamical system of finite entropy or finite dimension always has zero mean dimension. The mean dimension of the shift on $d$-cubes $([0,1]^d)^\mathbb{Z}$ is equal to $ d$ (\cite[Proposition 3.3]{LindenstraussWeiss2000MeanTopologicalDimension}). Computing the mean dimension is in general difficult and it is only known for few systems (\cite{tsukamoto2019mean, burguet2021mean}).  In the present paper  we investigate the mean dimension of the induced system $(\MM(X) ,T_*)$. Our main result states as below :

\begin{theomain}
For any topological system $(X,T)$ with positive topological entropy, the induced system $(\MM(X), T_*)$ has infinite topological mean dimension. Therefore, $$h_{top}(T)>0\Leftrightarrow \mdim(T_*)>0 \Leftrightarrow \mdim(T_*)=\infty.$$
\end{theomain}


Given a distance $d$ on $X$,  Lindenstrauss and Weiss \cite{LindenstraussWeiss2000MeanTopologicalDimension} has introduced the associated metric mean dimension $\mdim(X,d,T)$ as  the dynamical quantity corresponding  to the Minkowski dimension. They showed inter alia that the metric mean dimension is always larger than or equal to the associated  mean topological dimension. It is conjectured that the topological mean dimension is the infimum of the metric mean dimensions with respect to distances inducing the topology (this conjecture is known to hold true in some cases, see \cite{lindenstrauss2019double}).  A dynamical system of finite entropy has zero metric mean dimension for any metric compatible with the topology (\cite{LindenstraussWeiss2000MeanTopologicalDimension}). Together with the Main Theorem, we have the following corollary.

\begin{cor}Given  a topological system $(X,T)$  we have the following equivalences for any metric $D$ on $\MM(X)$ compatible with the weak-$*$ topology:
$$h_{top}(T)>0\Leftrightarrow \mdim(\MM(X),D,T_*)>0 \Leftrightarrow \mdim(\MM(X),D,T_*)=\infty.$$
\end{cor}

For an integer $a\geq 2$ we let $T_a$ be the $\times a$-map on the circle.  Kloeckner  \cite{Kl} proved  the following lower bound on the metric  mean dimension  of $T_a$ for the $p^{\rm th}$-Wasserstein metric:
$$\underline{\mdim}(\MM(X), W_p, (T_a)_*)\geq p(a-1).$$
In fact, the metric mean dimension of the $n$-power  of a system is always larger than or equal to $n$ times the metric mean dimension of the system (see  Inequality (\ref{power metric})). But  the $n$-power of $(T_a)_*$ is just $(T_{a^n})_*$, therefore 
without referring to our new results one can show directly from Kloeckner's theorem that the metric mean dimension of the induced map $(T_a)_*$ with respect to  the $p^{\rm th}$-Wasserstein metric is in fact infinite:
\begin{align*}
\underline{\mdim}(\MM(X), W_p, (T_a)_*)&\geq \frac{1}{n}\cdot\underline{\mdim}(\MM(X), W_p, (T_a)_*^n),\\
&\geq \frac{1}{n}\cdot\underline{\mdim}(\MM(X), W_p, (T_{a^n})_*),\\
&\geq p\cdot\frac{a^n-1}{n}\xrightarrow{n}\infty.
\end{align*}
Therefore, Kloeckner's theorem can be regarded as a special case of our Corollary.

For a topological system $(X,T)$ with positive entropy, the Main Theorem implies the mean dimension \footnote{See the definition in Section \ref{meansec}.} of $T_*$ with the scale $\epsilon$, denoted by $\mdim_{W_1}(T_*, \epsilon)$,  goes to infinity as $\epsilon$ tends to zero. The rate of divergence of $\mdim_{W_1}( T_*,  \epsilon)$ may be precised as follows:
\begin{thm}\label{thm:rate m}
Suppose $(X,T)$ is of positive entropy and Lipschitz. Then there exist $C, \alpha>0$ such that we have 
$$
 \forall 0<\epsilon<1,\ \ \mdim_{W_1}(T_*,  \epsilon)\ge {C \epsilon^{-\alpha}}.
$$
\end{thm}
In general, we can use the modulus of continuity of $T$ to estimate the rate of divergence of $\mdim(\MM(X), T_*, W_1, \epsilon)$, but for simplicity we only state the result for Lipshitz systems $(X,T)$. \\

The Corollary also implies  that the mean dimension of $T_*$ with respect to the Wasserstein metric $W_1$ associated to some given compatible metric $d$ on $X$ is infinite, i.e. 
$$\frac{h_{W_1}(T_*,\epsilon)}{\log \frac{1}{\epsilon}}\xrightarrow{\epsilon\rightarrow 0}+\infty.$$ We also investigate the rate of divergence of $h_{W_1}(T_*,\epsilon)$ as $\epsilon$ tends to zero for Lipschitz systems.
\begin{thm}\label{thm:rate h}
	Suppose $(X,T)$ is of positive entropy and Lipschitz. Then there exist $C, \alpha>0$ such that we have 
	$$
	 \forall 0<\epsilon<1,\ \  h_{W_1}(T_*,\epsilon)\ge C\epsilon^{-\alpha}.
	$$
\end{thm}

Under the same approach, our results (Main Theorem, Corollary, Theorem \ref{thm:rate m} and Theorem \ref{thm:rate h}) can be generalized to the action of congruent monotilable amenable groups, like $\Z^d, d\ge 2$. It seems that we could not apply directly the same method to the action of amenable groups which is not congruent monotilable. We thus wonder whether our results still hold for general amenable actions.

\subsection*{Organization of the paper}
In the next section we recall some basic facts on the  mean dimension. In the third section we give a  first direct proof of the Main theorem for full shifts on  $\{0,1\}^\mathbb{Z}$. This proof, which is independent from the proof  in  the general case, strongly uses the product structure on $\{0,1\}^\mathbb{Z}$. The mean dimension is shown to be infinite by embedding the shift on cubes of arbitrarily large dimension in the induced system $\left( \MM(\{0,1\}^{\mathbb Z}), \sigma_*\right)$. 
	Then we prove the main theorem in the general case. We work with the metric definition (\ref{metricdef})  of the topological mean dimension with respect to the Kantorovic-Rubinstein metric.
We use the theory of independence developed by   Weiss-Glasner \cite{glasner1995quasi}, Kerr-Li \cite{ker}  and others, which gives a \emph{weak} product structure in any topological system with positive entropy. By embedding generalized cubes of dimension $(k-1)n$ in the $k^{n}$-simplex we conclude the proof with a  Lebesgue's like theorem for these cubes. In Section 6 we establish the rates of convergence  for the mean dimension and entropy with respect to the Wasserstein distance given in Theorem \ref{thm:rate m} and Theorem \ref{thm:rate h}.
In the last section we study the mean dimension  of the action induced on the hyperspace of compact subsets of $X$. 

\section{Backgrounds}	\label{sec:background}
In the whole paper we let $(X,T)$ for a topological dynamical system, i.e. $X$ is compact metrizable and $T:X\rightarrow X$ a continuous map. 
We  denote by $d$ any metric  compatible with the topology on $X$.

\subsection{Dimension}
We first recall some standard concepts related to the dimension of the phase space $X$. 

\subsubsection{Order of a cover and Lebesgue's Lemma}
 We define the order $\rm ord (\mathcal A )$ of   a finite open cover $\mathcal{A}$  of $X$ as follows
	$$
	\text{\rm ord}(\mathcal{A}):=\sup_{x\in X} \sum_{A\in \mathcal{A}} 1_A(x)-1,
	$$

H. Lebesgue proved in \cite{Lebesgue} that cubes of different dimensions are not homeomorphic. We state below the key lemma of his approach.

\begin{lem}\cite{Lebesgue}
Let $\alpha$ be a finite open cover of the unit cube
$[0, 1]^n$. Suppose that there is no element of $\alpha$ meeting two opposite faces of $[0, 1]^n$.
Then we have  $$\rm ord (\alpha) \geq n.$$
\end{lem}
 
 Let $\Delta_k$ be the standard $k$-simplex, $k\in \mathbb N^*$, i.e. $$\Delta_k:=\left\{(x_i)_i\in [0,1]^{k}, \ \sum_ix_i=1 \right\}.$$ 
 To prove the Main Theorem, we will generalize Lebesgue's Lemma to $n$-direct  products of  the simplex $\Delta_k$ (see Lemma \ref{Leb}) in place of $[0,1]^n$.  

\subsubsection{Topological dimension}

 For two finite open covers $\mathcal{A}$ and $\mathcal{B}$ of $X$, we say that the cover $\mathcal{B}$ is \textit{finer} than the cover $\mathcal{A}$, and write $\mathcal{B}\succ \mathcal{A}$, if for every element of $\mathcal{B}$, one can find an element of $\mathcal{A}$ which contains it.  For a finite open cover $\mathcal{A}$  of $X$, we define the quantity 

	$$
	D(\mathcal{A}):=\min_{\mathcal{B}\succ \mathcal{A}} \text{ord}(\mathcal{B}).
	$$
	 For finite open covers $\mathcal{A}$ and $\BB$, we set the joint $\AA \vee \BB:=\{U\cap V: U\in \AA, V\in \BB \}$. It is easy to check that $D(\AA\vee \BB)\le D(\AA)+D(\BB)$. Clearly, if $\mathcal{B}\succ \mathcal{A}$ then $D(\mathcal{B})\ge D(\mathcal{A})$. 
	 
	 The \textit{(topological) dimension} of $X$ is defined by
	$$
	\text{dim}(X):=\sup_{\mathcal{A}} D(\mathcal{A}),
	$$
	where $\mathcal{A}$ runs over all finite open covers of $X$. 	For a non-empty compact  $X$,  the {\it stable topological dimension}
	of $X$ is given  by
	$$
	\stab(X):=\lim_{n\to \infty} \frac{\dim(X^n)}{n}=\inf_{n\to \infty} \frac{\dim(X^n)}{n},
	$$
	where $X^n$ denotes the $n$-direct product of $X$. The limit above exists by sub-additivity of the sequence $\{\dim(X^n) \}_{n\ge 1}$. Moreover, if $X$ is finite dimensional, then we have  either $\stab(X)=\dim(X)$ and $\dim(X^2)=2\dim(X)$ (the space $X$ is then said of {\it basic type} ) or $\stab(X)=\dim(X)-1$ and $\dim(X^2)=2\dim(X)-1$ (and $X$ is said of  {\it exceptional type})). M. Tsukamoto proved in \cite{tsukamoto2019mean} that the full shift on $X^{\mathbb  N}$ has topological mean dimension equal to $\stab(X)$ whenever $X$ is finite dimensional.

\subsubsection{Metric $\epsilon$-dimension and width}\label{detail}
Recall $d$ denotes a compatible distance on $X$. 
	 For a set $Z$ and $\epsilon>0$, a map $f:X\to Z$ is called \textit{$(d, \epsilon)$-injective} if $\diam(f^{-1}(z))<\epsilon$ for all $z\in Z$. The \textit{metric $\epsilon$-dimension} $\text{dim}_\epsilon (X, d)$ is defined by 
	$$
	\text{dim}_\epsilon (X, d)=\inf_Y \text{dim}(Y),
	$$
	where $Y$ runs over all compact metrizable spaces for which there exists a $(d, \epsilon)$-injective continuous map $f:X\to Y$. 
	
 	The metric $\epsilon$-dimension of $X$  is bounded from above by the topological dimension of $X$ and  is going non-increasingly to the topological dimension when $\epsilon$ goes to zero. 
 

A variant of the metric  $\epsilon$-dimension introduced by Gromov \cite{G} is the {\it $\epsilon$-width}, $\wdim_{\epsilon}(X, d)$
which is  the
smallest integer $n$ such that there exists an $\epsilon$-injective continuous map $f : X \rightarrow P$
from $X$ into some $n$-dimensional polyhedron $P$. We always have $$\dim_\epsilon(X,d)\leq \wdim_{\epsilon}(X, d)\leq 2\dim_{\epsilon}(X, d)+1.$$
Gromov showed that the unit ball $B$ of a $n$-dimensional Banach $(E,\|\cdot\|)$ space satisfies
$$\forall 1>\epsilon>0, \ \wdim_{\epsilon}(B,\|\cdot\|)=n.$$

 Motivated by a
question raised by Gromov \cite[p.334]{G}, Gournay \cite{gournay} and Tsukamoto \cite{Tsu088} obtained
some estimates for $\wdim_{\epsilon}(B_p, \ell_q)$ with $B_p$ being the unit ball in $(\mathbb R^n,\ell_p)$.	
	Gromov also wondered about the width of the  simplex $\Delta_n$.
	
\subsection{Entropy}
To fix the notations we  recall here Bowen's definition of topological 
entropy. For any positive integer  $n$ we consider the $n$-dynamical distance $d_n$ as follows, $$\forall x,y\in X,\ d_n(x,y)=\max_{0\le i\le n-1}d(T^ix, T^iy).$$

	Let $K\subset X$ and $\epsilon>0$. A subset $E$ of $X$ is said to be {\it $(n, \epsilon)$-separated}  if we have  $d_n(x,y)>\epsilon$ for any $x\neq y\in E$. Denote by $s_n(d, T, K, \epsilon)$  the largest cardinality of any $(n, \epsilon)$-separated subset of $K$. Define
	$$
	h_d(K, T, \epsilon)=\limsup_{n\to \infty} \frac{1}{n} \log s_n(d, T, K, \epsilon).
	$$
	We sometimes write $h_d(T, \epsilon)$ when $K=X$. Observe that  for all positive integer $n$ and all $\epsilon>0$ we have the following estimate on the entropy of the $n$-power $T^n$:
	\begin{align}\label{power entropy}
	h_d(T^n, \epsilon)\leq n \cdot h_d(T, \epsilon).
	\end{align} The topological entropy is the limit of  $h_d(T, \epsilon)$ when $\epsilon$ goes to zero.  This limit does not depend on the choice of the metric $d$ (this is no more the case of  the topological and metric mean dimensions). 

The entropy is homogeneous, i.e. $h_{top}(T^n)=n\cdot h_{top}(T)$ for all $n\in\mathbb N$ and $h_{top}(X,T)\geq h_{top}(Y,T)$ for any invariant closed subset $Y$ of $X$. Also the entropy of the full shift over a finite alphabet $\mathcal A$ is equal to $\log \sharp \mathcal A$.  A topological system $(X,T)$ is said to have  a \textit{horseshoe} when the full shift over $\{0,1\}^{\mathbb N}$ embeds in some power $(X,T^m)$. It follows from the above properties, that such systems have positive topological entropy. Many systems with positive topological entropy have a horseshoe, e.g.:
\begin{itemize}
\item coded subshifts,  
\item continuous interval and circle maps with positive entropy (by Misiurewicz horseshoe theorem \cite{MisHor}),
\item $C^{1+}$ surface diffeomorphisms (by Katok's horseshoe theorem \cite{katok80}),
\item ...
\end{itemize}
But there are also systems with positive topological entropy without horseshoes, even uniquely ergodic minimal subshifts.\\

Inspired by the works on emergence introduced by P. Berger, Y. Ji, E. Chen and X. Zhou \cite{chen2020} defined a topological order which estimates the superexponential growth of orbits at arbitrarily small scaled. Let   $W_{d_n}$ be the $1$-Wasserstein metric associated to $d_n$. They define the {\it entropy order} $\mathcal E(T)$ of $T$ by $$\mathcal E(T):=\lim_{\epsilon\rightarrow 0}\limsup_n \frac{\log\log s_n(W_{d_n}, T_*,\MM(X),\epsilon)}{n}.$$
 They showed that the entropy order  is always equal to the topological entropy by Jewett-Krieger theorem.

\subsection{Mean dimension}\label{meansec}
	
	Let $(X, T)$ be a topological  dynamical system, i.e. $X$ is a compact metrizable space and $T:X\circlearrowleft$ is a continuous map. The {\it mean dimension} of $(X, T)$ is defined by
	$$
	\mdim(X, T)=\sup_{\alpha} \lim_{n\to \infty} \frac{D(\bigvee_{i=0}^{n-1} T^{-i}\alpha )}{n},
	$$
	where $\alpha$ runs over all finite open covers of $X$. The existence of the limit  follows from  the sub-additivity of the sequence $\left(D(\bigvee_{i=0}^{n-1} T^{-i}\alpha)\right)_{n\ge 1}$. We put  also $\mdim(X,T,\alpha)=\lim_{n\to \infty} \frac{D(\bigvee_{i=0}^{n-1} T^{-i}\alpha )}{n}.$
	When there is no ambiguity we write also $\mdim(T)$ for $\mdim(X, T)$.

	We mention some basic properties of mean dimension. We refer to the book \cite{Coo05} for the proofs and further properties. 
	\begin{itemize}
		\item 
	For a metric $d$ on $X$ compatible with the topology, we have
		\begin{equation}\label{metricdef}
		\mdim(X,T)=\lim_{\epsilon \rightarrow 0} \mdim_d(T,\epsilon),
		\end{equation}
		where we let $\mdim_d(T,\epsilon)=\lim_{n\to \infty} \frac{\dim_{\epsilon}(X, d_n)}{n}.$
Observe the limit in $\epsilon$ in the definition 		(\ref{metricdef}) of $\mdim(X,T)$ is also a supremum over $\epsilon$. 
		\\
		
		\item If $(Y,T)$ is a subsystem of $(X,T)$, i.e. $Y$ is a closed $T$-invariant subset of $X$, then \begin{equation}\label{sub}\mdim(Y,T)\le \mdim(X,T).\end{equation}
		\item For $n\in \N$, \begin{equation}\label{power}\mdim(X,T^n)=n\cdot \mdim(X,T).\end{equation}
		\item For dynamical systems $(X_i,T_i)$, $1\le i\le n$, we have
		$$
		\mdim(X_1\times X_2 \times \dots \times X_n, T_1\times T_2 \times \dots \times T_n)\le \sum_{i=1}^n \mdim(X_i, T_i).
		$$
	\end{itemize}

	The {\it upper metric mean dimension} of the system $(X, d, T)$ is defined by
	$$
	\overline{\mdim} (X, d, T)=\limsup_{\epsilon\to 0} \frac{h_d(T, \epsilon)}{\log \frac{1}{\epsilon}}.
	$$
	Similarly, the {\it lower metric mean dimension} is defined by
	$$
	\underline{\mdim} (X, d, T)=\liminf_{\epsilon\to 0} \frac{h_d(T, \epsilon)}{\log \frac{1}{\epsilon}}.
	$$
	If the upper and lower metric mean dimensions coincide, then we call their common value the metric mean dimension of $(X, d, T)$ and denote it by ${\mdim} (X, d, T)$. Unlike the topological entropy, the metric mean dimension depends on the metric $d$. An important property of metric mean dimension is
	$$
	\mdim(X,T)\le \underline{\mdim} (X, d, T),
	$$
	for any metric $d$ \cite{LindenstraussWeiss2000MeanTopologicalDimension}.

It follows from the definitions and Inequality (\ref{power entropy}) that:
\begin{align}\label{power metric}
\forall n\in \mathbb N, \  \underline{\mdim} (X, d, T^n)\leq n\cdot  \underline{\mdim} (X, d, T).
\end{align}
	
\subsection{Embedding in induced system}	
	Let $(X,T)$ be a dynamical system. For any $n\in \mathbb N^*$
	we denote by  $(X^{(n)}, T^{(n)})$  the $n$-direct product of $(X,T)$. These products embed in $(\MM(X), T_*)$. Indeed we may find $n$ positive integers $k_1\ll \cdots \ll k_n$ such that for  any two subfamilies $I\neq J$  of $\{1,\cdots, n\}$ we have $\sum_{i\in I}k_i\neq \sum_{j\in J}k_j$. Then we define a dynamical embedding  $\pi:(X^{(n)}, T^{(n)})\rightarrow (\MM(X), T_*)$ as follows $$\forall (x_1,\cdots, x_n)\in X^{(n)}, \ \pi(x_1,\cdots, x_n)=
\frac{\sum_{i=1,\cdots, n}k_i\delta_{x_i}}{\sum_{i=1,\cdots, n}k_i },$$ where $\delta_x$ is the Dirac measure at $x$.\\

 The topological  entropy of $T^{(n)}$ is $n\cdot h_{top}(T)$. The  topological mean dimension  satisfies the same property as proved recently in \cite{Meanprod} :  $$\forall n\in \mathbb N, \ \mdim(X^{(n)}, T^{(n)})=n \cdot\mdim(X,T).$$ 
Therefore as the topological entropy  (resp. topological mean 
dimension) is a topological invariant which is smaller  for 
subsystems, we get $h_{top}(T_*)\geq n\cdot h_{top}(T)$ (resp. $\mdim
(T_*)\geq n\cdot \mdim(T)$) for all $n$. In particular if $T$ has positive 
topological entropy (resp. mean dimension) then $T_*$ has infinite 
topological entropy (resp. mean dimension). For the topological entropy this property was first noticed in \cite{SigAff}. \\

\section{A first direct proof for systems with a horseshoe}
We first prove the Main Theorem for systems with a horseshoe. By (\ref{power}) and (\ref{sub}) it is enough to show it for a full shift on a finite alphabet $\mathcal A$. 
Let $\sigma$ be the shift on $\AA^\N$, i.e. $(x_n)_{n\in \N} \mapsto (x_{n+1})_{n\in \N} $. For a finite word $a=(a_0,\cdots, a_{n-1})\in \AA^n$ with some integer $n$, we denote by $[a]$ the cylinder with preword $a$, i.e. 
$$
[a]:=\{x\in \AA^\N: x_i=a_i, \forall 0\le i\le n-1 \}.
$$

\begin{thm}\label{main thm 1}
For $\sharp \AA \ge 2$, we have $\mdim( \MM(\AA^\N), \sigma_*)=+\infty$.
\end{thm}
 Using the product structure of $\mathcal A^n$, we will embed the full shift over a simplex of arbitrarily large dimension in $(\MM(\mathcal A^\mathbb N ), \sigma_*)$.
\begin{proof}
Fix $m\in \mathbb N^*$. Let $\AA^m$ be the set of words of length $m$. Recall that $\Delta_{\sharp\mathcal A^m}$ denotes the $\sharp\mathcal A^m$-simplex, which can be seen as the set of probability measures on $\AA^m$: $$\Delta_{\sharp\mathcal A^m}:=\left\{ (x_i)_{i\in \AA^m}\in [0,1]^{\AA^m} : \sum_{i\in \AA^m} x_i=1 \right \}.$$
In this case we have  $\stab(\Delta_{\sharp\mathcal A^m})=\dim(\Delta_{\sharp\mathcal A^m})=\sharp \AA^m-1$. For any $a\in \Delta_{\sharp\mathcal A^m}$ we let $\mu_a$ be the probability measure on $\sharp 
\mathcal A^m$ with probability vector $a=(a^j)_{j\in \mathcal A^m}$, i.e. $$\mu_a(\{j\})=a^j \text{ for all } j\in \mathcal A^m.$$  We identify $(\mathcal A^m)^\mathbb N$ with $\mathcal A^\mathbb N$ by the natural homeomorphism   
$\pi:(\mathcal A^m)^\mathbb N\rightarrow\mathcal A^\mathbb N$, which associates  to $(x_j^i)_{i\in \{0,\cdots,m-1\}, j\in \N}$  the 
sequence $(y_k)_{k\in \mathbb N}$ defined by $y_{jm+i}=x_{j}^i$ for any $j\in \mathbb N$ and $i\in \{0,\cdots,m-1\}$. Then define $f_m:\Delta_{\sharp\mathcal A^m}^\mathbb N\rightarrow \mathcal M(\mathcal A^\mathbb N)$ as follows:
$$\forall a=(a_n)_{n\in \mathbb N}\in \Delta_{\sharp\mathcal A^m}^{\mathbb N}, \ f_m(a)=\pi_*(\mu_{a_1}\times\mu_{a_2}\times \cdots ),$$
where $\mu_{a_1}\times\mu_{a_2}\times \cdots$ denotes the direct product measure on $(\mathcal A^m)^\mathbb N$. Clearly $f_m$ is injective, continuous and for all  $a=(a_n)_{n\in \mathbb N}\in (\mathcal A^m)^{\mathbb N}$ we have with $\sigma$, $\sigma'$  and $\sigma''$ being respectively the shifts on $\AA^{\mathbb N}$, $(\mathcal A^m)^{\mathbb N}$ and $\Delta_{\sharp\mathcal A^m}^\mathbb N$  : 
\begin{align*}
f_m\circ \sigma'' (a)&=\pi_*(\mu_{a_2}\times\mu_{a_3}\times \cdots ),\\
&=\pi_*\circ \sigma'_*(\mu_{a_1}\times\mu_{a_2}\times \cdots ),\\
&= \sigma^m_*\circ \pi_*(\mu_{a_1}\times\mu_{a_2}\times \cdots ),\\
&=\sigma^m_*\circ f_m(a).
\end{align*}
In other terms, $f_m:(\Delta_{\sharp\mathcal A^m}^{\mathbb N}, \sigma'')\rightarrow (\MM(\AA^\mathbb N), \sigma_*^m)$ is a dynamical embedding. In particular  we get for all positive integers $m$:
\begin{align*}
 \mdim(\sigma_*)&=\frac{\mdim(\sigma_*^m)}{m},\\
&\geq \frac{\mdim(\sigma'')}{m},\\
&\geq  \frac{\stab(\Delta_{\sharp\mathcal A^m})}{m}.\\
&\geq \frac{\sharp \AA^m-1}{m}\xrightarrow{m\rightarrow \infty}+\infty.
\end{align*}
\end{proof}
From Theorem \ref{main thm 1}, the conclusion of the Main Theorem holds for any topological system $(X,T)$ with a horseshoe.
For a positive integer $a$ we recall $T_a$ denotes the $\times a$-map on the circle. For $a>1$ we have $h_{top}(T_a)=\log a>0$ and $T_a$ has then a horseshoe.
In particular the topological mean dimension of $T_a$ is infinite. Therefore we recover Kloeckner's result: 
$$\underline{\mdim}(\MM(X), W_1, (T_a)_*)\geq \mdim(\MM(X),(T_a)_*)=+\infty.$$

\section{Tools for the proof in the general case}

We need some preparation before proving  the Main Theorem for a general topological system with positive entropy. 

\subsection{Wasserstein distance}	
	
To estimate the topological mean dimension we use the metric definition (\ref{metricdef}) of the topological mean dimension for the Wasserstein distance $W_1$ on $\MM(X)$ (which is compatible with the weak-$*$ topology). The $p^{\text{th}}$-Wasserstein distance between two probability measures $\mu$ and $\nu$ in $\MM(X)$ is defined as
$$W_p(\mu, \nu)=\left(\inf_{\gamma\in \Gamma(\mu,\nu)}\int d(x,y)^p\, d\gamma(x,y)\right)^{1/p},$$
where $\Gamma (\mu ,\nu )$ denotes the collection of all measures on 
$X\times X$  with marginals $\mu$ and $\nu$  on the first and second factors respectively. In the present paper we only make use of $W=W_1$. The Kantorovich-Rubinstein dual representation of $W$ is given by 
$$\forall \mu,\nu\in \MM(X), \ W(\mu,\nu)=\sup_f\int f\,(d\mu-d\nu),$$
where the supremum holds over all $1$-Lipschitz functions $f:X\rightarrow \mathbb R$. \\

For a metric space $(Y,\rho)$ we let $\diam_\rho(Y)$ be the diameter of $Y$ with respect to the distance $\rho$. From the Kantorovich-Rubinstein dual representation of $W$ we  get easily 
$$\diam_W(\MM(X))\leq \diam_d(X).$$

We also recall below a useful  property of $ W$, which follows  from the Lipschitz property of the distance. 
\begin{lem}\label{ffact} 
 Let $\mu$ and $\nu$ be two measures in $\MM(X)$ respectively supported on two compact sets $S$ and $S'$. 
Then
$$W(\mu, \nu)\geq \mu(S\setminus S') d(S\setminus S',S').$$

\end{lem}

\begin{proof}
Write $\mu$ as $\mu=\lambda\mu'+(1-\lambda)\mu''$ with $\mu',\mu''\in \MM(X)$ and $\lambda\in [0,1]$ satisfying $\mu'(S\setminus S')=1$ and $\mu''( S')=1$. Note that we have $\lambda=\mu(S\setminus S')$. Then, as $d(\cdot, S')$ is $1$-Lipschitz, we get:
\begin{align*}
W(\mu,\nu)&\geq \int d(\cdot, S')\, (d\mu-d\nu),\\
&\geq \lambda\int d(\cdot, S')\, d\mu', \text{ because $\nu$ and $\mu''$ are supported on $S'$, }\\
&\geq \lambda d(S\setminus S', S')=\mu(S\setminus S')d(S\setminus S', S').
\end{align*}

\end{proof}

\subsection{Independence set}\label{indepe}
 For a subset $E$ of $\mathbb N$ we let  $E_n=E\cap [1,n]$ for all $n\in \mathbb N$. Then   we let  $\underline{d}(E)$ be the {\it lower  asymptotic density} of $E$ given by
$$\underline{d}(E)=\liminf _n\frac{\sharp E_n }{n}.$$ 
B. Weiss and E. Glasner \cite{glasner1995quasi}, then D. Kerr and J. Li showed positive entropy implies some weak product structure of dynamical system. 

\begin{lem}[\cite{ker} Proposition 3.9 (2)]\label{keli}
Let $(X,T)$ be a topological system with positive topological entropy. There exist  non empty open subsets $U_0,U_1$ of $X$ with $\overline{U_0}\cap \overline{U_1}= \emptyset$ and $I\subset \mathbb N$ with $\underline{d}(I)>0$ such that  for any finite subset $J$ of $I$ and for any $\zeta:J\rightarrow \{0,1\}$ we have 
$$\bigcap_{j\in J}T^{-j}U_{\zeta(j)}\neq \emptyset.$$
\end{lem}
Conversely the existence of such a pair $U_0$ and $U_1$ clearly implies that $(X,T)$ has positive entropy. The set $I$ is called an \textit{independence set } and  $(U_0,U_1)$ is an {\it IE-pair}. Existence of IE-pairs follows from a combinatorial result due to N. Sauer \cite{Sauer} and S. Shelah \cite{Shelah}. This result was also a key point in the "non-ergodic" proof of $[h_{top}(T)=0]\Rightarrow [h_{top}(T_*)=0]$ by B. Weiss and E. Glasner \cite{glasner1995quasi}.

 Given $(X,T)$ with $h_{top}(T)>0$ we fix from now the   IE-pair $(U_0,U_1)$ and the independence set $I$ as in Lemma \ref{keli}.  For all $m\in \mathbb N$ we let $$q_m=\left\lfloor \frac{m \cdot\underline{d}(I)}{2}\right\rfloor$$
where $\lfloor x \rfloor$ is the largest integer $n\le x$ and we consider 
$$I^m=\left\{k\in \mathbb N, \ \sharp [km,(k+1)m[\cap I >q_m \right\}.$$

As in the proof of Theorem \ref{main thm 1} we will bound from below the mean dimension of $T_*^m$ to show the Main Theorem  and then take the limit in $m$.  We need the following lower bound on the  lower  asymptotic density of $I_m$.
\begin{lem}\label{comb}
For any $m\ge 1$, we have
$$\underline{d} (I^m)\geq \frac{\underline{d}(I)}{2}.$$
\end{lem}

\begin{proof}
For $n\in \mathbb N$, we have $$\sharp I_{mn}\leq m \sharp I^m_n+q_m (n-\sharp I^m_n), $$
therefore by taking the limit  in $n$ we get:
\begin{align*}
\underline{d}(I)&\leq \frac{m-q_m}{m}\underline{d}(I^m)+\frac{q_m}{m},\\
&\leq \underline{d}(I^m)+\frac{\underline{d}(I)}{2}.
\end{align*}

\end{proof}

\subsection{Simplex and generalized cube}
Recall that  $\Delta_k$ is the standard $k$-simplex of dimension $k-1$. A {\it $\ell$-face} of the simplex $\Delta_k$ is $$\{ (x_i)_{i\in I} \in [0, 1]^I: \sum_{i\in I} x_i =1 \}$$ for some $I\subset \{1, 2, \dots, k\}$ with $\sharp I=\ell$. Cleqrly, any $\ell$-face is affinely homeomorphic to $\Delta_\ell$. The {\it opposite face} $\bar{F}$ of a face $F=\{ (x_i)_{i\in I} \in [0, 1]^I: \sum_{i\in I} x_i =1 \}$ is $$\{ (x_i)_{i\in I^c} \in [0, 1]^{I^c}: \sum_{i\in I^c} x_i =1 \}$$ where $I^c=\{1, 2, \dots, k\} \setminus I$. The opposite face of a $(k-1)$-face $F$ is just the vertex of the simplex which does not belong to $F$. Notice also that for $m$ $(\le k)$ distinct $(k-1)$-faces $(F_j)_{1\le j\le m}$, the intersection ${\cap_{1\le j\le m} F^j}$ is a $(k-m+1)$-face.

We denote by  $v(\Delta_k)$  the set of vertices of $\Delta_k$ and by $*$ the center of $\Delta_k$, i.e. the point with coordinates $(1/k,\cdots,1/k)$.  Note that $\Delta_1$ is a compact interval. The polyhedron given by the $n$-product $\Delta_k ^n$ will be called a \textit{generalized cube}.  
Let  $F$ be a face  of $\Delta_k$.  For $i\in \{1,\cdots, n\}$ we let 
$F_i$ be the face of  $\Delta_k^n$ given by   $$F_i=\underbrace{\Delta_k\times \cdots \Delta_k}_{i-1}\times F\times \underbrace{\Delta_k\times \cdots \Delta_k}_{n-i}.$$ 
The {\it boundary} of $\Delta_k^n$, denoted by $\partial\Delta_k^n$, is the union of $F_i$ for all $(k-1)$-faces $F$ and $1\le i\le n$. The {\it interior} of $\Delta_k^n$ is $\Delta_k^n \setminus \partial\Delta_k^n$.

We introduce an adapted  notion of separation of faces for generalized cubes.
\begin{df}
A cover $\alpha$ of $\Delta_k^n$ is said to be \textbf{separating} if for any $i\in \{1,\cdots,n\}$ and for any subfamily $(U_j)_{j\in J}$ of $\alpha$ with  $(k-1)$-faces $(F^j)_{j\in J}$ of $\Delta_k$ satisfying $F^j_i\cap U_j\neq \emptyset$ for all $j\in J$, the intersection  set 
$ \bigcap_{j\in J}U_j $ and  the opposite face $\left(\overline{\cap_{j\in J} F^j}\right)_i$  are disjoint.
\end{df}

\subsection{Generalized Lebesgue's lemma}

We need a Lebesgue's like Lemma for  generalized cubes. 


\begin{lem}[Generalized Lebesgue's lemma]\label{Leb}
Any separating open cover $\alpha$ of $\Delta_k^n$  has order larger than or equal to $nk$. 
\end{lem}
\begin{proof} We argue by contradiction by considering a separating open cover $\alpha$ of $\Delta_k^n$ with order less than $nk$. 
Let $(f_U)_{U\in \alpha}$ be a partition of unity associated to $\alpha$, i.e. $f_U:\Delta_k^n\rightarrow [0,1]$ is a continuous function supported on $U$ for any $U\in \alpha$ and $\sum_{U\in \alpha}f_U=1$.  
We define a function $\phi^\alpha=(\phi_1^\alpha,\cdots, \phi^\alpha_n):\alpha\rightarrow \Delta_k^n$ as follows:  for any $U\in \alpha$ and any $i\in\{1,\cdots,n\}$ we let 
\begin{equation*}
\phi^\alpha_i(U):=
\begin{cases}
*, &\text{ if } U\cap F_i=\emptyset \text{ for any $(k-1)$-face $F$}, \\
&\\
\underbrace{\{0\}\times \cdots \{0\}}_{i-1}\times \overline{F}\times \underbrace{ \{0\}\times \cdots  \{0\}}_{n-i}, &\text{ for some  $(k-1)$-face $F$ with $F_i\cap U\neq \emptyset$}.
\end{cases}
\end{equation*}
Note that $\phi^\alpha$ is not uniquely defined in this way as $U$ may intersect different faces of $F_i$, but it does not affect the proof. 
  We define a function $g:\Delta_k^n\circlearrowleft$ by   $$\forall
  x\in \Delta_k^n, \ \ g(x)=\sum_{U\in \alpha} \phi^\alpha(U)f_U(x).$$

  Since the order of $\alpha$ is less than $nk$ the function  $g$ takes values in a finite union of $(nk-1)$-dimensional convex subsets of $\Delta_k^n$.  Therefore there is  a point 
$\overline{x}$ in the interior of $\Delta_k^n$ which does not lie in the image of $g$. Let $r: \Delta_k^n\setminus \{\overline{x}\}\rightarrow \partial\Delta_k^n $ be a retraction map on the boundary $\partial\Delta_k^n$ of $\Delta_k^n$.

Notice that if $g(x)$ does not lie in the interior of $\Delta_k^n$ then there is $i\in \{1,\cdots, n\}$ and  a family $\mathcal{F}$ of $(k-1)$-faces such that :
\begin{itemize}
\item for each element $U$ of $\alpha$ containing  $x$ there is a $(k-1)$-faces $F\in \mathcal{F}$ with $U\cap F_i\neq \emptyset$, 
\item $g(x)$ belongs to  $\left(\overline{\cap_{F\in \mathcal{F}} F}\right)_i$. 
\end{itemize} As the cover $\alpha$ is assumed to be separating, we have in particular $g(x)\neq x$ for any $x\in \partial\Delta_k^n$. 
Therefore the restriction of  $g$  to the boundary of $\Delta_k^n$ has no fixed point. Finally the map $r\circ g :\Delta_k^n:\circlearrowleft$ has also no fixed point.  This contradicts Brower's fixed point theorem.
\end{proof}

\subsection{A multi-affine embedding of the generalized cube $\Delta_k^n$ in the simplex $\Delta_{k^n}$ }\label{secemb}
To simplify the notations we let $\llbracket n \rrbracket:=\{1,\cdots, n\}$ for any positive integer $n$.
The generalized cube $\Delta_k^n$ may be embedded  in the simplex $\Delta_{k^n}$  by the following map:
\begin{equation}
\begin{split}
\Theta: \Delta_{k}^{n} & \to \Delta_{k^n},\\
(t_m)_{m\in \llbracket n \rrbracket} & \mapsto \sum_{\mathbf i= (i_m)_{m\in \llbracket n \rrbracket} }t_{\mathbf i}{e_{\mathbf i}} ,
\end{split}
\end{equation}

where $e_{\mathbf i}$ is a standard basis of $\R^{k^n}$ and  $t_{\mathbf i}=\prod_{m\in\llbracket n \rrbracket}t_{m,{i_m}}$ with $t_m=(t_{m,i})_{i\in \llbracket k \rrbracket}\in \Delta_{k}$ for any $m\in \llbracket k \rrbracket$.

\begin{figure}[h]
\begin{center}
\includegraphics[scale=0.22]{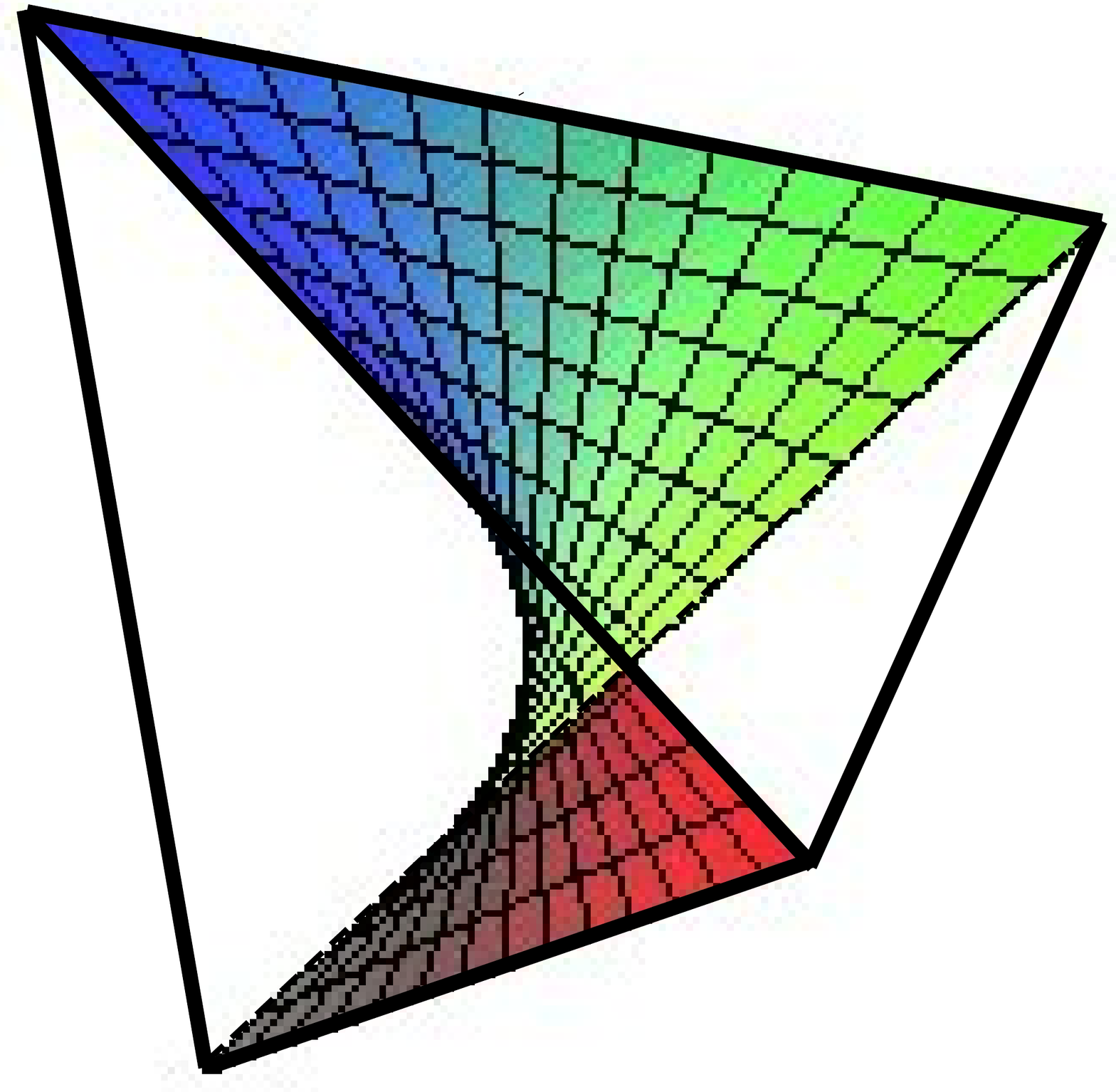}
\end{center}
\caption{\label{etiquette} The hyperbolic paraboloid inside the simplex $\Delta_4$.}
\end{figure}

Clearly $\Theta$ is continuous and injective. It is not hard to check that the embedding $\Theta$ is not affine, but {\it multi-affine}, i.e. it is affine separately in each variable $t_m$.  In particular the image $S$ of $\Delta_k^n$ under $\Theta$ is not convex, but its faces may still be joined by segment lines inside  $S$.  For example if one takes $n=k=2$, then the image of $\Theta$ is a hyperbolic paraboloid, which is a doubly ruled two-dimension surface in $\Delta_4$ (see Figure \ref{etiquette}).

\section{Proof of Main Theorem}

In this section, we prove our main theorem.
To this end, we shall  prove the mean dimension of the induced map $T_*^m$ grows exponentially fast in $m$, whenever $(X,T)$ has positive topological entropy.  Recall that $W$ denotes the $1$-Wasserstein distance.  We  let $W_n^m$ be the $n$-dynamical $1$-Wasserstein
distance for $T_*^m$, i.e. $$\forall \mu, \nu\in \MM(X), \ W_n^m(\mu, \nu)=\max_{0\leq k<n}W(T_*^{km}\mu,T_*^{km}\nu).$$
Remark that $W_n^1$ may differ from the Wasserstein distance $W_{d_n}$ associated to the dynamical distance $d_n$. In general we only have $W_n^1\leq W_{d_n}$ (see Proposition 2.1 in  \cite{chen2020}).

Suppose $(X,T)$ is a dynamical system with positive topological entropy. Let $(U_0,U_1)$ be the IE-pair and $I$ be the independence set of $(X,T)$ as in Lemma \ref{keli}.  We state below the main proposition towards the Main Theorem.
\begin{prop}\label{lastt}
For any $m\ge 1$ there is some $\epsilon_m>0$ such that
\begin{equation}\label{prop}\dim_{\epsilon_m}(\MM(X), W_n^m)\geq 2^{q_m}\sharp I^m_n.
\end{equation}
\end{prop}

We first conclude the proof of the Main Theorem assuming Proposition \ref{lastt}.
\begin{proof}[Proof of the  Main Theorem]
By taking $n\to +\infty$ in (\ref{prop}) we  get the following inequality according to Lemma \ref{comb}:
$$\lim_{n\to +\infty} \frac{\dim_{\epsilon_m}(\MM (X), W^m_n)}{n} \geq  2^{q_m-1}\cdot \underline{d}(I).$$
Then by (\ref{metricdef}) we have 
\begin{align*}
\mdim(\MM(X), T_*^m)&\geq \lim_{n\to \infty} \frac{\dim_{\epsilon_m}(\MM (X), W^m_n)}{n},\\
&\geq 2^{q_m-1}\cdot \underline{d}(I).
\end{align*}
Finally by using the formula (\ref{power}) of the mean dimension of powers, we conclude that (recall that $q_m$ is equal to $\lfloor  \frac{m\cdot\underline{d}(I)}{2} \rfloor$):
\begin{align*}
\mdim(\MM(X), T_*)&= \frac{\mdim(\MM(X),T_*^m)}{m},\\
&\geq \frac{\mdim_{W}(T_*^m,\epsilon_m)}{m},\\
&\geq \lim_{n\to \infty} \frac{\dim_{\epsilon_m}(\MM (X), W^m_n)}{n},\\
&\geq \frac{2^{q_m-1}}{m}\cdot \underline{d}(I)\xrightarrow{m\rightarrow +\infty}+\infty.
\end{align*}
This completes the proof.
\end{proof}
The  end of this section is devoted to the proof of Proposition \ref{lastt}. Fix $m\ge 1$. For any $k\in I^m$ we let $E_k$ be the $q_m$ first integers in $[km, (k+1)m[\cap I$ and we consider a bijection  $$\psi_k:\llbracket 2^{q_m} \rrbracket \rightarrow \{0,1\}^{E_k}.$$
Let $n\ge 1$. Since $I$ is the independent set of $(X,T)$, we can pick, for each sequence $\mathbf{i}=(i_k)_{k\in I_n^m}\in \llbracket 2^{q_m} \rrbracket ^{I_n^m}$, a point $x_{\mathbf{i}}$ with  $$x_{\mathbf{i}}\in\displaystyle{\bigcap_{ \substack{k\in I_n^m,\\ \ell\in E_k}} T^{-\ell}U_{a_\ell^k}} \text{ where } \psi_k(i_k)=(a_\ell^k)_{\ell\in E_k}.$$
Let us denote by $\delta_x$ the Dirac measure at $x\in X$. Let $K_n$ inside $\MM(X)$ be the $(2^{q_m})^{\sharp I_n^m}$-simplex given by the  convex hull of all $\delta_{x_{\mathbf i}}$ for $\mathbf{i}\in \llbracket 2^{q_m} \rrbracket ^{I_n^m}$. 
To prove Proposition \ref{prop} it is enough to show that
$$\dim_{\epsilon_m}(K_n, T_*^m)\geq  2^{q_m} \sharp I_n^m.$$

We consider the affine homeomorphism 
\begin{equation}
\begin{split}
\Psi: \Delta_{2^{q_m \sharp I_n^m}} & \to K_n,\\
(t_\mathbf i)_{\mathbf i\in \llbracket 2^{q_m} \rrbracket ^{I_n^m}} & \mapsto \sum_{\mathbf i }t_{\mathbf i}\delta_{x_{\mathbf i}}.
\end{split}
\end{equation}
We let $\Xi=\Psi\circ \Theta$ with $\Theta$ being the multi-affine embedding of $\Delta_{2^{q_m}}^{ \sharp I_n^m}$ in $\Delta_{2^{q_m \sharp I_n^m}}$ defined in Subsection \ref{secemb}. Notice that $\Xi$ is also a multi-affine embedding. We denote the image of $\Xi$ by $L_n$. 
For $E\subset \MM(X)$, we put $S_E:=\bigcup_{\mu\in E}  \supp(\mu)$. Note that $S_{K_n}$ consists of finitely many points $x_{\mathbf i}$ for $\mathbf i\in \llbracket 2^{q_m} \rrbracket ^{I_n^m}$.  We have the following decomposition of the measure in $L_n$.

\begin{lem}\label{decomposition}
Let $1\le \ell\le k-1$. Let $F$  be a $\ell$-face of $\Delta_{2^{q_m}}$ and $\overline{F}$ be its opposite $(2^{q_m}-\ell)$-face. Let $i\in I_n^m$. For any  $\mu\in L_n$,  there is $\lambda\in [0,1]$, $\mu'\in \Xi(F_i)$ and $\mu''\in \Xi(\overline{F}_i)$ with 
$$\mu=\lambda  \mu'+(1-\lambda)\mu'',$$
where $\mu(S_{\Xi(F_i)})=\lambda$ and $\mu(S_{\Xi(\overline{F}_i)})=1-\lambda$.
\end{lem}

\begin{proof}Let $F$ be a $\ell$-face of $\Delta_{2^{q_m}}$ and $i\in I_n^m$. Suppose $\mu=\Xi(t)$ for some $t\in \Delta_{2^{q_m}}^{I_n^m}$. 
Since the simplex $\Delta_{2^{q_m}}$ is the convex hull of $F$ and its opposite $(2^{q_m}-\ell)$-face $\overline{F}$, the generalized cube $\Delta_{2^{q_m}}^{I_n^m}$ is the convex hull of $F_i$ and $\overline{F}_i$.
Then we write $t=\lambda  t'+(1-\lambda)t''$ for some $\lambda\in [0,1]$, $t'\in F_i$ and $t''\in \overline{F}_i$.  As $\Xi$ is affine in the $i^{th}$ coordinate, we conclude that 
\begin{align*}
\mu=\lambda  \Xi(t')+(1-\lambda)\Xi(t'').
\end{align*}
\end{proof}

In the following we identify $L_n$ and the  generalized cube  $\Delta_{2^{q_m}}^{I_n^m}$ via the embedding $\Xi$. In particular we will talk about the faces of $L_n$ as the images by $\Xi$ of the faces of $\Delta_{2^{q_m}}^{I_n^m}$.
 For $m\in \mathbb N$, by continuity of $T$, there is $\gamma_m>0$ such that $$\left[d(x,y)<\gamma_m\right]\Rightarrow [d(T^kx, T^ky)<d(U_0,U_1),\forall k=0,\cdots, m-1].$$

\begin{lem}\label{tech}
Let  $\mu\in L_n$ and  $i\in I_n^m$. For any face $F$ of $\Delta_{2^{q_m}}$ we have 
\begin{enumerate}
\item 
$$W_n^m(\mu, F_i)\geq \mu(S_{\overline{F}_i}) \gamma_m,$$ 
 in particular, 
\begin{align*}
W_n^m\left(F_i, \overline{F}_i\right) 
\geq \gamma_m.
\end{align*}
\item 
$$W_n^m(\mu, F_i)\leq \diam_d(X) \mu(S_{\overline{F}_i}).$$ 
\end{enumerate}
\end{lem}
\begin{proof}
We write $
\mu$ as  $$\mu=\beta\mu'+(1-\beta)\mu''$$ with $\mu''\in F_i$, $
\mu'\in \overline{F}_i$ and $\beta=\mu\left(S_{\overline{F}_i}\right)$ as in Lemma \ref{decomposition}. 
\begin{enumerate}
\item  

If $x_\mathbf{i}$ and $x_{\mathbf i'}$ belong respectively to $S_{\overline{F}_i}$ and $ S_{F_i}$, then there exists $0\leq k<m$ with $T^{im+k}x_{\mathbf i}\in U_j$ and $T^{im+k}x_{\mathbf i'}\in U_{j'}$, $j\neq j'\in \{0,1\}$. By definition of $\gamma_m$ we have $$d(T^{im}x_{\mathbf i}, T^{im}x_{\mathbf i'})\geq \gamma_m.$$
It follows that 
$$
d(T^{im}S_{\overline{F}_i}, T^{im}S_{F_i})\ge \gamma_m.
$$
Consequently by Lemma \ref{ffact}  we have for all $\nu\in F_i $ :
\begin{align*}
W_n^m(\mu,\nu )&\geq W(T^{im}\mu, T^{im}\nu),\\
&\geq \beta d(T^{im}S_{\overline{F}_i}, T^{im}S_{F_i}), \\
&\geq \beta \gamma_m.
\end{align*}
Therefore $W_n^m(\mu, F_i)\geq \mu(S_{\overline{F}_i}) \gamma_m.$
Finally, notice that if $\mu \in \overline{F}_i$, then $\beta=1$. This implies that $W_n^m\left(F_i, \overline{F}_i\right) 
\geq \gamma_m.$
\item 
 Since $\mu''\in  F_i$, we have
\begin{align*}
W_n^m(\mu,F_i )&\leq 
W_n^m(\mu, \mu''),\\
&\leq \beta \cdot W_n^m(\mu',\mu''),\\
&\leq \beta \cdot \diam_d(X).
\end{align*}
\end{enumerate}
\end{proof}

Now we let 
\begin{equation}\label{eplison}
\epsilon_m= \frac{\gamma_m^2}{\diam_d(X)2^{q_m+1}}.
\end{equation}

\begin{lem}\label{inter}
Let $i\in I^m_n$. Then for any $\mu\in L_n$ and any family $\mathcal F$ of $(2^{q_m}-1)$-faces of $\Delta_{2^{q_m}}$ with $W^m_n(\mu,F_i)<\epsilon_m$ for all $F\in \mathcal F$, we have 
$$W^m_n(\mu,\bigcap_{F\in \mathcal F}F_i)\leq \frac{\gamma_m}{2}.$$
\end{lem}
\begin{proof}
By Lemma \ref{tech} (1), we have 
$$\forall F\in \mathcal F, \ \mu(S_{\overline{F}_i})\leq \frac{\epsilon_m}{\gamma_m}.$$
Then according to Lemma \ref{tech} (2), we have 
\begin{align*}
W^m_n(\mu,\bigcap_{F\in \mathcal F}F_i)&\leq  \diam_d(X)\mu\left(S_{\left(\overline{\bigcap_{F\in \mathcal F}F}\right)_i}\right),\\
&\leq \diam_d(X)\sum_{F\in \mathcal F}\mu(S_{\overline{F}_i}), \text{ as $S_{\left(\overline{\bigcap_{F\in \mathcal F}F}\right)_i}\subset \bigcup_{F\in \mathcal F}S_{\overline{F}_i}$, }\\
&\leq \diam_d(X) \frac{\epsilon_m}{\gamma_m} \sharp \mathcal F,\\
&\leq \diam_d(X) \frac{\epsilon_m}{\gamma_m} 2^{q_m}=\frac{\gamma_m}{2}.
\end{align*}

\end{proof}

\begin{lem}\label{coversep}
Any cover $\alpha$ of $L_n$ with $W^m_n$-diameter less than $\epsilon_m$  is separating.  
\end{lem}
\begin{proof} Fix $i\in I_n^m$. 
According to Lemma \ref{inter}, if $(U_j)_{j\in J}$ is  a subfamily of $\alpha$ with $\bigcap_{j\in J}U_j\neq \emptyset$ and $F^j_i\cap U_j\neq \emptyset$, $j\in J$, for some $(2^{q_m}-1)$-faces $F^j$ of $\Delta_{2^{q_m}}$, then  $\bigcap_{j\in J}U_j$ lies in the  
$\gamma_m/2$-neighborhood of $\bigcap_{j\in J}F^j_i$. Noting that $\overline{\bigcap_{j}F^j}$ is the opposite face of $\bigcap_{j}F^j$ we have in the other hand by Lemma \ref{tech} (1):
\begin{align*}
W_n^m\left(\bigcap_{j}F^j_i, \left(\overline{\bigcap_{j}F^j}\right)_i\right)& 
\geq \gamma_m.
\end{align*}
Therefore $\bigcap_jU_j$ and $\left(\overline{\bigcap_{j}F^j}\right)_i$ are disjoint.

\end{proof}

We can now conclude the proof of Proposition \ref{lastt}.

\begin{proof}
	Let $m,n\ge 1$. Let $\epsilon_m$ be defined as \eqref{eplison}.
	Suppose $f:L_n\rightarrow Z$ is a $(W_n^m,\epsilon_m)$-injective map with $Z$  a compact metrizable space of finite dimension. Then 
	there is an open cover $\alpha$ of $L_n$ with diameter less than $\epsilon_m$ and with order less than or equal to $\dim(Z)$. According to Lemma \ref{coversep}, the cover $\alpha$ of $L_n$ is separating. From Lemma \ref{Leb} the order of $\alpha$ is larger than $2^{q_m}\sharp I^m_n$. Then we conclude that 
	\begin{align*}
	\dim_{\epsilon_m}(\MM(X), W_n^m) &\geq \dim_{\epsilon_m}(L_n, W_n^m), \\
	&\geq 2^{q_m}\sharp I^m_n.
	\end{align*} 
	This completes the proof of Proposition \ref{lastt},  therefore of the Main Theorem.
\end{proof}

\begin{proof}[Proof of Theorem \ref{thm:rate m}]
Suppose $T$ is $K$-Lipschitz, so that we can take $$\gamma_m=K^{-m}\cdot \frac{d(U_0, U_1)}{2}.$$ Then we have 
$$
\epsilon_m=c_1 e^{-mc_2},
$$
where $c_1=\frac{d(U_0, U_1)^2}{4\diam_d(X)}$ and $c_2=2\log K+\frac{\underline{d}(I)\log 2}{2}$. It follows that
\begin{align*}
\mdim_{W}(T_*,  \epsilon_m)&\ge \frac{1}{m} \mdim_W( T_*^m, \epsilon_m),\\
&\ge \lim_{n\rightarrow +\infty}\frac{\dim_{\epsilon_m}(\MM(X),W^m_n)}{n},\\
&\ge \lim_n\frac{2^{q_m}\sharp I^m_n}{n},\\
&\geq 2^{m\frac{\underline{d}(I)}{2}-2}\underline{d}(I).
\end{align*}

Since $ \epsilon_{m+1}/ \epsilon_m$ is bounded, we can find $C>0$ such that with $\alpha=\frac{\underline{d}(I)\log 2}{2c_2}$ it holds that :

$$\forall 0<\epsilon<1, \ \ \mdim_{W}(T_*,  \epsilon)\ge C\epsilon^{-\alpha}.$$
\end{proof}

\section{Rate of divergence of $h_W(T_*,\epsilon)$}\label{sec:rate}
In this section, we discuss about the precise rate of divergence of $h_W(T_*,\epsilon)$. We first give a rough upper bound on the ratio $h_W(T_*,\epsilon)$.  For $\epsilon>0$ and a metric space $(Z, \rho)$, let $s(Z,\rho, \epsilon)$ be the smallest cardinality of a $\epsilon$-covering set in $Z$ with respect to $\rho$. The (upper) Minkowski dimension $\dim_d(Z)$ of $(Z,d)$ is defined as $\dim_d(Z)=\limsup_{\epsilon\rightarrow 0}\frac{\log s(Z,\rho, \epsilon) }{\log\frac{1}{\epsilon}}$. 

F. Bolley, A. Guillin, and C. Villani proved  the following estimate on $s(\MM(X),W,\epsilon)$:
\begin{thm}\label{vil}(Theorem A1 in \cite{bol07})
For $\epsilon>0$ small enough, we have:
$$ s(\MM(X),W,\epsilon)\leq \left(\frac{1}{\epsilon}\right)^{s(X,d,\epsilon/2)}.$$ 
\end{thm}

\begin{lem}\label{lem:h 1}
For any $\alpha >\dim_d(Z)$, there is $C>0$ such that
$$\forall 0<\epsilon<1, \ \ 
h_W(T_*,\epsilon)\le C \epsilon^{-\alpha}.
	$$

\end{lem}
\begin{proof}
Obviously, we have $h_W(T_*,\epsilon)\leq \log  s(\MM(X),W,\epsilon/2)$ for any $\epsilon>0$. By Theorem \ref{vil}, we get  $$h_W(T_*,\epsilon)\leq  s(X,d,\epsilon/4)\log \frac{2}{\epsilon}\le  C \epsilon^{-\alpha},$$ for some  constant $C$.
\end{proof}

Moverover, when $T$ is $K$-Lipschitz, we can refine the inequality in Lemma \ref{lem:h 1}. Indeed one easily checks in this case that $T_*$ is $K$-Lipschitz with respect to the Wasserstein distance. Then we have  
$$h_W(T_*,\epsilon)\leq \sup_{\mu\in \MM(X)}s(B_W(\mu, K\epsilon),W,\epsilon),$$
where $B_W(\mu, K\epsilon)$ denotes the $W$-ball of radius $K\epsilon$ centered at $\mu$. However we do not know any better estimate of  $\sup_{\mu\in \MM(X)}s(B_W(\mu, K\epsilon),W,\epsilon)$ than the one given in Theorem \ref{vil}.\\

To estimate  $h_W(T_*,\epsilon)$ from below, we first prove a technical lemma.
\begin{lem}\label{lem:s_n}
	Let $(X,T)$ be a dynamical system of positive entropy. With the notations of Subsection \ref{indepe}
	 we have $$s_n\left(W,T_*^m,\MM(X),\frac{\gamma_m}{2^{q_m}}\right)\ge 2^{2^{q_m}\sharp I^m_n}. $$
\end{lem}
\begin{proof}
	Define 
	$$
	H:=\{ (t_k)_{k\in \llbracket 2^{q_m} \rrbracket}: \exists \ell\ge 1, \exists 1\le i_1<i_2<\dots <i_\ell\le 2^{q_m}, \text{ s.t. } t_{i_1}=t_{i_2}=\dots=t_{i_\ell}=1/\ell   \},
	$$
	which is a subset of $\Delta_{2^{q_m}}$. Moreover, note that $\sharp H=2^{2^{q_m}}$. Let $H_n$ be the image of $H^{I_n^m}$ via $\Xi$. Clearly, $\sharp H_n= 2^{2^{q_m}\sharp I^m_n}$. We will show that $H_n$ is a $(n, \frac{\gamma_m}{2^{q_m}})$-separated set of $(\MM(X),W,T_*^m)$. Let $\mu\not=\nu \in H_n$. Then there exists $i\in \llbracket I^m_n \rrbracket$ and a face $F$ of $\Delta_{2^{q_m}}$ such that either $\mu(S_{\overline{F_i}})\ge 2^{-q_m}, \nu(S_{{F_i}})=1$ or $\nu(S_{\overline{F_i}})\ge 2^{-q_m}, \mu(S_{{F_i}})=1$. Without loss of generality, we assume that it is  the former case. It follows 
	by Lemma \ref{ffact}  that 
	\begin{align*}
	W_n^m(\mu,\nu )&\geq W(T^{im}\mu, T^{im}\nu),\\
	&\geq \mu(S_{\overline{F_i}}) d(T^{im}S_{\overline{F}_i}, T^{im}S_{F_i}), \\
	&\geq 2^{-q_m} \gamma_m.
	\end{align*}
	It means that $\mu$ and $\nu$ are $(n, \frac{\gamma_m}{2^{q_m}})$-separated. This completes the proof.
\end{proof}

Now we are in a position to prove  Theorem \ref{thm:rate h}.

\begin{proof}[Proof of Theorem \ref{thm:rate h}]
	Suppose $T$ is of positive entropy and $K$-Lipschitz. Let $U_0$ and $U_1$ be an IE-pair and let $I$ be the associated independence set. Then we pick $$\gamma_m=K^{-m}\cdot \frac{d(U_0, U_1)}{2}.$$ Thus we have
$$
\frac{\gamma_m}{2^{q_m}}\ge c_1 e^{-mc_2},
$$
where $c_1=\frac{d(U_0, U_1)}{2}$ and $c_2=\log K+\frac{\underline{d}(I)\log 2}{2}$. 
It follows from Lemma \ref{lem:s_n} that
\begin{align*}
h_W\left(T_*^m, \frac{\gamma_m}{2^{q_m}}\right) &\ge \frac{1}{m} h_W\left(T_*^m,\frac{\gamma_m}{2^{q_m}}\right),\\
&\ge  \frac{c_2\underline{d}(I)}{\log(1/\epsilon_m) } \left(\frac{c_1}{{\gamma_m}/{2^{q_m}}}\right)^{-\frac{\underline{d}(I)\log2}{2c_2}}.
\end{align*}
Since $ \frac{\gamma_{m+1}/2^{q_{m+1}}}{\gamma_m/2^{q_m}}$ is bounded, we conclude that  for any  $\alpha<\frac{\underline{d}(I)\log 2}{2 \log K+\underline{d}(I)\log 2}$,  there is $C>0$ such that 
	$$
	 \forall 0<\epsilon<1,\ \  h_{W}(T_*,\epsilon)\ge C\epsilon^{-\alpha}.
	$$

\end{proof}

\begin{rem}In the case that  $(X,T)$ is of positive entropy and non-Lipschitz, we may have $\lim_{\epsilon\to 0}\frac{\log h_W(T_*,\epsilon)}{\log \frac{1}{\epsilon}}=0$, e.g.  when $X$ has  zero topological dimension by Lemma \ref{lem:h 1}. As an explicit example, one can consider the full shift on $K^{\mathbb Z}$ with $K$ being an infinite zero-dimensional compact space endowed with a metric of zero Minkowski dimension.  \end{rem}

When $(X,T)$ is  Lipschitz with positive entropy, the lower Minkowski dimension of $(X,T,d)$ is positive (see e.g. \cite{Mil}). Then by combining Lemma \ref{lem:h 1} and Theorem \ref{thm:rate h} we get:
\begin{coro}
	Suppose $(X,T)$ is of positive entropy and Lipschitz with $X$ a compact metric space with finite Minkowski dimension. Then there exists $\alpha>1$ such that	we have 
	$$\forall 0<\epsilon<1,\ \ 
	\alpha^{-1}\le \frac{\log h_W(T_*,\epsilon)}{\log \frac{1}{\epsilon}}\le \alpha.
	$$
\end{coro}

\section{Mean dimension of the dynamics on the hyperspace}
For a compact metrizable space we let $2^X$ be the hyperspace of $X$, i.e. the set of closed non empty subsets of $X$ endowed with the Hausdorf distance. For a continuous map $T:X\rightarrow X $ we let $T_\mathcal K:2^X\rightarrow 2^X$ be the dynamical system induced by $(X,T)$ on $2^X$. These induced systems have also been studied in \cite{Sigmun} and \cite{glasner1995quasi}. In particular it still holds true that 
$$[h_{top}(T)>0]\Rightarrow [h_{top}(T_\mathcal K)=\infty],$$
but there are zero entropy system $(X,T)$ with $h_{top}(T_\mathcal K)>0$. One may wonder about the relations between the topological entropy of $T$ and the mean dimension (either metric or topological) of $T_\mathcal K$. 
We present below two examples. 

\subsection{Example of $h_{top}(T)=0$ and $\mdim(T_\mathcal K)>0$}
Let $K$ be a compact metric space of finite dimension. Let $T: \cup_{\Z\cup \{\infty\} } K \to \cup_{\Z\cup \{\infty\} } K$ to be $x_n \mapsto x_{n+1}$ for $n\in \Z$ and $x_\infty \mapsto x_\infty$ where $\Z\cup \{\infty\}$ is seen as the one-point compactification of the integers. Then it is clear that $h_{top}(T)=0$. On the other hand, it is not hard to verify that the hyperspace system contains the full shift $(K^\Z, \sigma)$. It follows that $$\mdim(2^{\cup_{\Z\cup \{\infty\} } K}, T_\mathcal{K})\ge \text{stabdim}(K).$$
In particular, we get $\mdim(2^{\cup_{\Z\cup \{\infty\} } K}, T_\mathcal{K})=\infty$ when $K=[0,1]^\mathbb{N}$.

\subsection{Example of $h_{top}(T)>0$ and $\mdim(T_\mathcal K)=0$}
Let $X$ be a  zero dimensional compact metric space. As well known, $2^X$ is a zero dimensional compact metric space. Thus for any continuous transformation $T: X\to X$, we have $\mdim(2^X, T_\mathcal{K})=0$. In particular, for a full shift $(\mathcal{A}^\Z, \sigma)$ with $1<\sharp \mathcal{A}<\infty$, we have $h_{top}(\sigma)=\log \sharp \AA>0$ and $\mdim(\sigma_\mathcal K)=0$.





\bibliographystyle{alpha}
\bibliography{universal_bib}

\def\cprime{$'$} \def\cprime{$'$}
\begin{thebibliography}{BGV07}

\bibitem[BGV07]{bol07}
Fran\c{c}ois Bolley, Arnaud Guillin, and C\'{e}dric Villani.
\newblock Quantitative concentration inequalities for empirical measures on
  non-compact spaces.
\newblock {\em Probab. Theory Related Fields}, 137(3-4):541--593, 2007.

\bibitem[BS75]{Sigmun}
Walter Bauer and Karl Sigmund.
\newblock Topological dynamics of transformations induced on the space of
  probability measures.
\newblock {\em Monatsh. Math.}, 79:81--92, 1975.

\bibitem[BS21]{burguet2021mean}
David Burguet and Ruxi Shi.
\newblock Mean dimension of continuous cellular automata.
\newblock {\em arXiv preprint arXiv:2105.09708, Isr. J. Math., to apear}, 2021.

\bibitem[Coo05]{Coo05}
Michel Coornaert.
\newblock {\em Dimension topologique et syst\`emes dynamiques}, volume~14 of
  {\em Cours Sp\'ecialis\'es [Specialized Courses]}.
\newblock Soci\'et\'e Math\'ematique de France, Paris, 2005.

\bibitem[Gou11]{gournay}
Antoine Gournay.
\newblock Widths of {$\ell^p$} balls.
\newblock {\em Houston J. Math.}, 37(4):1227--1248, 2011.

\bibitem[Gro99]{G}
Misha Gromov.
\newblock Topological invariants of dynamical systems and spaces of holomorphic
  maps. {I}.
\newblock {\em Math. Phys. Anal. Geom.}, 2(4):323--415, 1999.

\bibitem[GW95]{glasner1995quasi}
Eli Glasner and Benjamin Weiss.
\newblock Quasi-factors of zero-entropy systems.
\newblock {\em Journal of the American Mathematical Society}, 8(3):665--686,
  1995.

\bibitem[JQ21]{Meanprod}
Lei Jin and Yixiao Qiao.
\newblock Mean dimension of product spaces: a fundamental formula, 2021.

\bibitem[Kat80]{katok80}
A.~Katok.
\newblock Lyapunov exponents, entropy and periodic orbits for diffeomorphisms.
\newblock {\em Inst. Hautes \'{E}tudes Sci. Publ. Math.}, (51):137--173, 1980.

\bibitem[KL07]{ker}
David Kerr and Hanfeng Li.
\newblock Independence in topological and {$C^*$}-dynamics.
\newblock {\em Math. Ann.}, 338(4):869--926, 2007.

\bibitem[Klo13]{Kl}
Beno\^{\i}t Kloeckner.
\newblock Optimal transport and dynamics of expanding circle maps acting on
  measures.
\newblock {\em Ergodic Theory Dynam. Systems}, 33(2):529--548, 2013.

\bibitem[Leb11]{Lebesgue}
Henri Lebesgue.
\newblock Sur la non-applicabilit\'{e} de deux domaines appartenant
  respectivement \`a des espaces \`a {$n$} et {$n+p$} dimensions.
\newblock {\em Math. Ann.}, 70(2):166--168, 1911.

\bibitem[LT19]{lindenstrauss2019double}
Elon Lindenstrauss and Masaki Tsukamoto.
\newblock Double variational principle for mean dimension.
\newblock {\em Geometric and Functional Analysis}, pages 1--62, 2019.

\bibitem[LW00]{LindenstraussWeiss2000MeanTopologicalDimension}
Elon Lindenstrauss and Benjamin Weiss.
\newblock Mean topological dimension.
\newblock {\em Israel J. Math.}, 115:1--24, 2000.

\bibitem[Mil]{Mil}
John Milnor.
\newblock Introductory dynamics lectures.
\newblock {\em http://www.math.stonybrook.edu/~jack/DYNOTES/}.

\bibitem[Mis79]{MisHor}
Micha\l Misiurewicz.
\newblock Horseshoes for mappings of the interval.
\newblock {\em Bull. Acad. Polon. Sci. S\'{e}r. Sci. Math.}, 27(2):167--169,
  1979.

\bibitem[Sau72]{Sauer}
N.~Sauer.
\newblock On the density of families of sets.
\newblock {\em J. Combinatorial Theory Ser. A}, 13:145--147, 1972.

\bibitem[She72]{Shelah}
Saharon Shelah.
\newblock A combinatorial problem; stability and order for models and theories
  in infinitary languages.
\newblock {\em Pacific J. Math.}, 41:247--261, 1972.

\bibitem[Sig78]{SigAff}
Karl Sigmund.
\newblock Affine transformations on the space of probability measures.
\newblock In {\em Dynamical systems, {V}ol. {III}---{W}arsaw}, pages 415--427.
  Ast\'{e}risque, No. 51. 1978.

\bibitem[Tsu08]{Tsu088}
Masaki Tsukamoto.
\newblock Macroscopic dimension of the {$l^p$}-ball with respect to the
  {$l^q$}-norm.
\newblock {\em J. Math. Kyoto Univ.}, 48(2):445--454, 2008.

\bibitem[Tsu19]{tsukamoto2019mean}
Masaki Tsukamoto.
\newblock Mean dimension of full shifts.
\newblock {\em Israel Journal of Mathematics}, 230(1):183--193, 2019.

\bibitem[YJ]{chen2020}
Xiaoyao~Zhou Yong~Ji, Ercai~Chen.
\newblock Entropy and emergence of topological dynamical systems,
  https://arxiv.org/abs/2005.01548.

\end{thebibliography}

\end{document}